\documentclass[12pt,a4paper]{article}
\usepackage{graphicx}
\usepackage{amsfonts}
\usepackage{amssymb}
\usepackage{amsthm}
\usepackage[centertags]{amsmath}
\usepackage{newlfont}

\usepackage{color}
\usepackage{graphicx,color}
\usepackage{amsmath, amssymb, graphics}

\def\r{\mathbb R}
 
 \def\s{\mathbb S}

\newtheorem{theorem}{Theorem}[section]

\newtheorem{proposition}[theorem]{Proposition}
\newtheorem{definition}[theorem]{Definition}

\title{Ruled surfaces of generalized self-similar solutions of the mean curvature flow}

\author{Rafael L\'opez\footnote{Partially supported by the grant no. MTM2017-89677-P, MINECO/AEI/FEDER, UE}\\
 Departamento de Geometr\'{\i}a y Topolog\'{\i}a \\
 Universidad de Granada\\
 18071 Granada, Spain\\
\texttt{rcamino@ugr.es}}

\date{}
\begin{document}
\maketitle

\begin{abstract} In Euclidean space, we investigate surfaces whose mean curvature $H$ satisfies the equation $H=\alpha\langle N,\mathbf{x}\rangle+\lambda$, where $N$ is the Gauss map, $\mathbf{x}$ is the position vector and $\alpha$ and $\lambda$ are two constants. There surfaces generalize self-shrinkers and self-expanders of the mean curvature flow. We classify the ruled surfaces and the translation surfaces, proving that they are cylindrical surfaces. 
\end{abstract}

\section{Introduction and statement of the results}

In Euclidean space $\r^3$, the theory of self-shrinkers, and to a lesser extent also  expander-shrinkers, has developed a great interest in the last decades. Self-shrinkers  are surfaces $M$ characterized by the equation
\begin{equation}\label{sh}
H(\mathbf{x})=-\frac{1}{2}\langle N(\mathbf{x}),\mathbf{x}\rangle,\quad \mathbf{x}\in M,
\end{equation}
where $N$ is the Gauss map of $M$ and $\langle,\rangle$ is the Euclidean metric of $\r^3$. Here $H$ is the trace of the second fundamental form so the mean curvature of  a sphere of radius $r>0$ is $2/r$ with respect to the inward normal. Analogously, self-expanders   satisfy \eqref{sh} but replacing the factor $-1/2$ by $1/2$. Self-shrinkers play an important role in the study of   the mean curvature flow because  they correspond to rescaling solutions   of an early time slice. Moreover,  self-shrinkers provide information about the behaviour of the  singularities of the  flow. The literature in the topic of self-shrinkers is sufficiently large to   give a summary. We address the reader to \cite{cm,lee,il}  and references therein as a first approach.  

There are very few explicit examples of self-shrinkers. First  examples  are vector planes, the sphere  of radius $2$ centered at the origin and the   round cylinder of radius $\sqrt{2}$ whose axis passes through  the origin.  Other examples appear when one assumes some type of invariance of the ambient space.  A first family of surfaces are those one that  are  invariant by a uniparametric group of translations. In such a case,  the equation \eqref{sh} reduces in an ordinary differential equation that describes the curvature of the generating planar curve (\cite{al,aaa,ha,km}. A second type or surfaces are the helicoidal surfaces, including rotational surfaces.  Rotational and  helicoidal   shrinkers were studied in   \cite{ha2,km}.

 Self-shrinkers   can be also seen as weighted minimal surfaces in the context of manifolds with density: see \cite{gr,mo}.   Let $e^\varphi$ be a   positive density  in $\r^3$, where $\varphi$ is a smooth function in $\r^3$. We use the density $e^\varphi$  as a weight for  the surface and the volume area. Let $M$ be a surface and  let  $\Phi:(-\epsilon,\epsilon)\times M\rightarrow\r^3$ be a compactly supported variation of $M$ with $\Phi(0,-)=M$. Denote by      $A_\varphi(t)$ and $V_\varphi(t)$    the weighted area  and the enclosed weighted  volume  of $\Phi(\{t\}\times M)$, respectively. The formulae of the first variation of $A_\varphi(t)$ and $V_\varphi(t)$ are  
$$A'_\varphi(0)=-\int_M H_\varphi \langle N,\xi\rangle\  dA_\varphi,\quad V_\varphi'(0)=\int_M \langle N,\xi\rangle\ dA_\varphi,$$
where $\xi$ is the variational vector field of $\Phi$ and 
$$H_\varphi=H-\langle N,\nabla\varphi \rangle$$
 is called the {\it weighted mean curvature}.    Consequently,  $M$ is a critical point of the functional area $A_\varphi$ if and only if $H_\varphi=0$. If we choose  the function $\varphi$ as 
\begin{equation}\label{fi}
\varphi(\mathbf{x})=\alpha \frac{|\mathbf{x}^2|}{2},\quad \mathbf{x}\in\r^3,
\end{equation}
the expression of $H_\varphi$ is   $H_\varphi=H(\mathbf{x})-\alpha\langle N,\mathbf{x}\rangle$. In particular,    self-shrinkers
are     critical points of the weighted area functional $A_\varphi$ for $\alpha=-1/2$. In case that we seek   critical poins of $A_\varphi$ for arbitrary  variations preserving the weighted volume, we deduce that  the function   $H_\varphi$ is   constant. After this motivation, and for the function $\varphi$ given in \eqref{fi}, we generalize the notion of self-shrinkers.

\begin{definition} Let $\alpha,\lambda\in\r$. A surface $M$ in $\r^3$ is said to be an $\alpha$-self-similar solution of constant $\lambda$ if 
\begin{equation}\label{ll}
H(\mathbf{x})=\alpha\langle N(\mathbf{x}),\mathbf{x}\rangle+\lambda,\quad \mathbf{x}\in M.
\end{equation}
\end{definition}
The case $\alpha=0$ corresponds with the surfaces  of constant mean curvature. This situation will be discarded in this paper and we will assume $\alpha\not=0$.    Examples of solutions of equation \eqref{ll} are  again   spheres centered at the origin and round cylinders whose axis passes through   the origin, but now, and in both cases,  the radius is arbitrary. Also  affine planes are   solutions of \eqref{ll}.

When $\alpha=-1/2$ in equation \eqref{ll}, self-shrinkers of constant $\lambda$ were  studied independently by  Cheng and Wei (\cite{cw}) and  McGonagle and   Ross (\cite{mgr}). Since then, and if $\alpha=-1/2$, these surfaces have received the interest for geometers:   \cite{ch,cow,cw0,gu,ross,zfc}.
 
 Let us point out that   the equation \eqref{ll} is invariant by linear  isometries of $\r^3$. So  if $A:\r^3\rightarrow\r^3$ is a linear isometry and $M$ is an $\alpha$-self-similar solution of constant $\lambda$, then $A(M)$ satisfies \eqref{ll} with the same constants $\alpha$ and $\lambda$.  
 We also notice  that a surface can be a solution of \eqref{ll} for different values of $\alpha$ and $\lambda$. For example, the sphere of radius $2$ centered at the origin satisfies \eqref{ll} for $(\alpha,\lambda)=(-1/2,0)$ and $(\alpha,\lambda)=(1/2,2)$.

   In this paper we investigate   $\alpha$-self-similar solutions of constant $\lambda$ under the geometric assumption that $M$ is a ruled surface. A {\it ruled surface} is a surface that is the union of a one-parameter family of straight lines.   A ruled surface  can be parametrized locally by 
\begin{equation}\label{ru}
X(s,t)=\gamma(s)+t\beta(s),
\end{equation}
where $t\in\r$ and $\gamma,\beta:I\subset\r\rightarrow\r^3$ are smooth curves with $|\beta(s)|=1$ for all $s\in I$. The curve $\gamma(s)$ is called the directrix of the surface and a line having $\beta(s)$ as direction vector is called a ruling of the surface.   In case that  $\gamma$ reduces into  a point,    the surface is called  conical.   On the other hand, if  the rulings are all parallel to a fixed direction ($\beta(s)$ is constant), the surface  is called   {\it cylindrical}. It is clear that a ruled surface is cylindrical if and only if it is invariant by a uniparametric  group of translations, namely, along the direction of $\beta$. 

In this paper, we classify all ruled surfaces that are solutions of the $\alpha$-self-similar equation \eqref{ll}.

\begin{theorem}\label{t1}
Let $M$ be an $\alpha$-self-similar solution of constant $\lambda$. If $M$ is   a ruled surface, then $M$ is a cylindrical surface.
\end{theorem}

This result was proved in \cite{an} for self-shrinkers. Cylindrical surfaces with  $\alpha=-1/2$ and $\lambda\not=0$ were classified in \cite{ch}. The key in the proof of Theorem \ref{t1} is that, by means of the parametrization \eqref{ru},    equation \eqref{ll}  is a polynomial equation on the variable $t$, whose coefficients are functions on the variable $s$. Thus all these coefficients must vanish and from here, we will prove the result. The proof of Theorem \ref{t1} will be carried out in Section \ref{sec2}.

Our second result refers to the study of the solutions \eqref{ll} by the method of separation of variables. We stand for $(x,y,z)$ the canonical coordinates of $\r^3$. Let  $M$ be    a graph $z=u(x,y)$, where $u$ is a function defined in some domain $\r^2$.     If $M$ is an $\alpha$-self-similar solution of constant $\lambda$, then $u$ is a solution of  
\begin{equation}\label{sh4}
\mbox{div}\frac{Du}{\sqrt{1+|Du|^2}}=\alpha\,\frac{u-\langle (x,y),Du\rangle}{\sqrt{1+|Du|^2}}+\lambda.
\end{equation}
This equation is a quasilinear elliptic equation and, as one can expect from the minimal surfaces theory,   that it is hard to find explicit solutions of \eqref{sh4}. A first approach to solve this equation is  by means of the method of separation of variables. The idea is to  replace a   function $u(x,y)$ by  a function that is   the sum of two   functions, each one depending in one  variable. Thus, we consider $u(x,y)= f(x)+g(y)$, where  $f:I\subset\r\rightarrow\r$ and $g:J\subset\r\rightarrow\r$ are smooth functions.  In such a case, we prove the following result.

\begin{theorem} \label{t2}
If $z=f(x)+g(y)$ is an $\alpha$-self-similar solution of constant $\lambda$, then $f$ or $g$ is a linear function. In particular, the surface is   cylindrical. Moreover, and after a linear isometry of $\r^3$, we have $g(y)=0$ and $f(x)$ satisfies the equation
\begin{equation}\label{cee}
\frac{f''(x)}{(1+f'(x)^2)^{3/2}}=\alpha\frac{-xf'(x)+f(x)}{\sqrt{1+f'(x)^2}}+\lambda.
\end{equation}
\end{theorem}

The proof of this result will done in Section \ref{sec3}. Since the function $u(x,y)$ is the sum of two functions of one variable,   equation \eqref{sh4} leaves to be a partial differential equation and converts into an ordinary differential equation where appears the derivatives of the functions $f$ and $g$. Then we will successfully  solve completely the solutions of equation \eqref{sh4}.

\section{Classification of ruled surfaces}\label{sec2}

 In this section we prove Theorem \ref{t1}. The proof   consists to assume that the ruled surface is parametrized by \eqref{ru} and that the rulings are not parallel. In such a case,  we shall prove that  an $\alpha$-self-similar solution of constant $\lambda$ must be a plane,  which it is a cylindrical surface. Let us observe that a plane is a ruled surface and that can be parametrized by \eqref{ru} but being $\beta$  a non-constant curve. 
 
 On the other hand, the cylindrical surfaces that satisfy \eqref{ll} are the one-dimensional version of the $\alpha$-self-similar solutions. Indeed, after a linear isometry of the ambient space,   we assume that the rulings are parallel to the $y$-line. We parametrize the surface as  $X(s,t)=\gamma(s)+t(0,-1,0)$, where $\gamma$ is a curve contained in the $xz$-plane $\Pi$ parametrized by arc-length. Then \eqref{ll} is 
 \begin{equation}\label{kk}
 \kappa_\gamma(s)=\alpha\langle\mathbf{n}(s),\gamma(s)\rangle+\lambda,
 \end{equation}
 where $\kappa_\gamma$ is the curvature of $\gamma$ as planar curve in $\Pi$ and $\{\gamma'(s),\mathbf{n}\}$ is a positive orthonormal frame in $\Pi$ for all $s\in I$.
 
Consider a ruled surface parametrized by  $X(s,t)=\gamma(s)+t\beta(s)$ as in \eqref{ru}, $|\beta(s)|=1$,  and suppose that $\beta$ is a not constant curve.   Since $|\beta(s)|=1$, the curve $\beta$ is a curve in the unit sphere $\s^2=\{\mathbf{x}:|\mathbf{x}|=1\}$.  Without loss of generality, we  assume that   $\beta$ is parametrized by arc-length, $|\beta'(s)|=1$ for all $s\in I$. From now, we drop the dependence of the variable of the functions.  In particular, $\mathcal{B}=\{\beta,\beta',e_3:=\beta\times\beta'\}$ is an orthonormal basis of $\r^3$ and  
  \begin{equation}\label{bbb}
\begin{split}
\beta''&=-\beta+\Theta\, e_3 ,\quad \Theta=(\beta,\beta',\beta'').\\
e_3'&=-\Theta \beta'.
\end{split}
\end{equation}
Here we stands for $(u,v,w)$ the determinant of the vectors $u,v,w\in\r^3$. 
 
 Firstly, we need to obtain an expression of  equation \eqref{ll} for the   parametrization   $X(s,t)$. We denote with the subscripts $s$ and $t$ the derivatives of a function with respect to the variables $s$ and $t$. Let us notice that $X_t=\beta$ and $X_{tt}=0$. The coefficients of the first fundamental form with respect to $X$ are $E=|X_s|^2$, $F=\langle X_s,X_t\rangle$ and $G=|X_t|^2=1$. Set $W=EG-F^2$. Consider the unit normal vector field $N=(X_s\times X_t)/\sqrt{W}$. Then equation \eqref{ll} is
\begin{equation}\label{sh3}
(X_s,X_t,X_{ss})-2fF(X_s,X_t,X_{st})=\alpha W ( X,X_{s}, X_t) +\lambda\, W^{3/2}.
\end{equation}

 A first case to discuss is when $X(s,t)$ is a conical surface.

\begin{proposition} \label{pr1}
Planes are the only  conical surfaces that are   $\alpha$-self-similar of constant $\lambda$.  
\end{proposition}

\begin{proof}
Suppose that $M$ is a conical surface parametrized by    $X(s,t)=p_0+t\beta(s)$,   where $p_0\in\r^3$ is a fixed point .  Then
  $F=0$, $W=t^2$ and  equation \eqref{sh3} is 
$$t^2(\beta',\beta,\beta'')-\alpha t^3(p_0,\beta',\beta)-\lambda t^3=0.$$
This is a polynomial equation in the variable $t$, where the  coefficients depend only on the variable $s$. Thus we deduce $(\beta,\beta',\beta'')=0$ and $\alpha(p_0,\beta,\beta')-\lambda=0$. Since $\beta$ is a curve in the unit sphere $\s^2$ parametrized by arc-length,   it is not difficult to conclude from $(\beta,\beta',\beta'')=0$ that $\beta$ is a great circle of $\s^2$. This proves that the surface is a plane containing the point $p_0$, proving the result.
\end{proof}

 From now, we assume that the ruled surface is not conical, that is, $\gamma$ is not a constant curve.  The next step of the proof of Theorem \ref{t1} is to choose a suitable parametrization of the ruled surface.  In a ruled surface, it is  always possible to take a (not unique)   special parametrization that consists in taking for $\gamma$ a curve orthogonal to the rulings, that is, $\langle\gamma(s),\beta(s)\rangle=0$ for all $s\in I$.

As we pointed out in the introduction,   we can assume that $\gamma(s)$ is a curve perpendicular to the rulings of the surface, that is,  $\langle\gamma'(s),\beta(s)\rangle=0$ for all $s\in I$. 

 The derivatives of $X$ with respect to $s$ and $t$ are
$$X_s=\gamma'(s)+t\beta'(s),\quad X_t=\beta(s)$$
$$X_{ss}=\gamma''(s)+t\beta''(s),\quad X_{st}=\beta'(s),\quad X_{tt}=0.$$
Then $F=\langle X_s,X_t\rangle=\langle \gamma',\beta\rangle=0$, $G=1$  and   
\begin{equation}\label{ee}
E=\langle X_s,X_s\rangle=|\gamma'|^2+2t\langle\gamma',\beta'\rangle+t^2.
\end{equation}
The unit normal vector field  is  
$$N=\frac{\gamma'\times\beta-t e_3}{\sqrt{E}}.$$
Equation \eqref{sh3} is now
\begin{equation}\label{e-rrr}
\mathcal{L}=
\alpha E\left((\gamma',\beta,\gamma)-t\langle e_3,\gamma\rangle\right)+\lambda E^{3/2},
\end{equation}
where
$$\mathcal{L}=-(\beta,\beta',\beta'')t^2+t\left((\beta',\beta,\gamma'')+(\gamma',\beta,\beta'')\right)+(\gamma',\beta,\gamma'').$$
Using \eqref{bbb}, we write this equation as
$$\mathcal{L}=-\Theta\, t^2-t\left(\langle e_3,\gamma''\rangle+\Theta\langle \gamma',\beta'\rangle\right)+(\gamma',\beta,\gamma'').$$

We distinguish the cases $\lambda=0$ and $\lambda\not=0$.
\begin{enumerate}
\item Case $\lambda=0$. We see \eqref{e-rrr} as a polynomial on the variable $t$, which  is  of degree $3$ by the expression of $E$ in \eqref{ee}. From the  coefficient for $t^3$, we have 
$$\alpha \langle e_3,\gamma\rangle=0.$$
Since $\alpha\not=0$, we deduce $ \langle e_3(s),\gamma(s)\rangle=0$ for all $s\in I$. Then $\gamma(s)$ belongs the plane determined by $\beta(s)$ and $\beta'(s)$. Let
\begin{equation}\label{bb}
\gamma(s)=u(s)\beta(s)+v(s)\beta'(s)
\end{equation}
for some smooth functions $u=u(s)$ and $v=v(s)$.  Now  equation \eqref{e-rrr} is $\mathcal{L}=
\alpha E (\gamma',\beta,\gamma)$.  Matching the   coefficients on $t$ of degree $2$, $1$ and $0$, we obtain, respectively, 
\begin{eqnarray*}
\Theta&=&-\alpha (\gamma',\beta,\gamma)\\
\langle e_3,\gamma''\rangle+\Theta\langle \gamma',\beta'\rangle&=&-2\alpha\langle\gamma',\beta'\rangle(\gamma',\beta,\gamma)\\
(\gamma',\beta,\gamma'')&=&\alpha|\gamma'|^2(\gamma',\beta,\gamma).
\end{eqnarray*}
Using  the basis $\mathcal{B}$ and equation \eqref{bb},  we calculate the velocity of $\gamma(s)$, obtaining 
\begin{equation}\label{b4}
\gamma'=(u'-v)\beta+(u+v')\beta'+v\Theta\, e_3.
\end{equation}
Since $\langle\gamma',\beta\rangle=0$, we have $u'-v=0$.  From this expression of $\gamma'$ in combination with   \eqref{bbb}, we obtain $(\gamma,\gamma',\beta)=v^2\Theta$. Then the three above identities become
\begin{eqnarray}
\Theta&=&-\alpha v^2\Theta\label{s1}\\
\langle e_3,\gamma''\rangle+\Theta\langle \gamma',\beta'\rangle&=&-2\alpha v^2 \langle\gamma',\beta'\rangle\Theta\label{s2}\\
(\gamma',\beta,\gamma'')&=&\alpha v^2|\gamma'|^2 \Theta.\label{s3}
\end{eqnarray}
From \eqref{s1}, 
$$(1+\alpha v^2)\Theta=0.$$

We discuss two cases.
\begin{enumerate}
\item Case $\Theta=0$. As in Proposition \ref{pr1}, the curve $\beta(s)$ describes a great circle of $\s^2$. In particular, $e_3=\beta\times\beta'$ is a unit constant vector  orthogonal to the plane $P$ containing $\beta$. Moreover, from \eqref{bb}, $\langle\gamma(s),e_3\rangle=0$ for all $s\in I$.   Thus 
$$\langle X(s,t),e_3\rangle=\langle\gamma(s)+t\beta(s),e_3\rangle=\langle\gamma(s),e_3\rangle=0.$$
This proves that the surface is part of the plane  $P$.   
\item  Case $\Theta\not=0$. Then 
\begin{equation}\label{av}
1+\alpha v^2=0.
\end{equation}
 In particular, $v$ is a non-zero constant function and $v'=0$. Moreover,  from \eqref{bbb} and \eqref{b4}, 
\begin{equation}\label{gg}
\begin{split}
\gamma'&=u\beta'+v\Theta e_3,\\
\gamma''&=-u\beta+v(1-\Theta^2)\beta'+(u\Theta+v\Theta')e_3.
\end{split}
\end{equation}
From these expressions, we   compute the terms of the identity \eqref{s2}, obtaining
$$2u\Theta+v\Theta'=-2\alpha uv^2\Theta.$$
Due to \eqref{av}, the above equation is simply $v\Theta'=0$. Since $v\not=0$ from \eqref{av}, we have shown that  $\Theta$ is a constant function. 

We now compute the terms of the identity \eqref{s3}. Because $\Theta$ is constant, and taking into account \eqref{av} and \eqref{gg}, we find
$$(\gamma',\beta,\gamma'')=(v^2-u^2)\Theta-v^2\Theta^3,$$
$$\alpha v^2|\gamma'|^2\Theta=-(u^2+v^2\Theta^2)\Theta.$$
Thus  \eqref{s3} reduces $v^2\Theta=0$,   obtaining a contradiction.

\end{enumerate}

\item Case   $\lambda\not=0$. Squaring the equation \eqref{e-rrr} 
\begin{equation}\label{b00}
\Big((\mathcal{L}-\alpha E((\gamma',\beta,\gamma)-t\langle e_3, \gamma\rangle)\Big)^2-\lambda^2 E^3=0.
\end{equation}
Set $\Gamma=|\gamma'|^2$. Equation \eqref{b00}  is polynomial equation on $t$ of degree $6$ whose coefficients are functions on the variable $s$,  hence all them must vanish. The coefficients for $t^6$ and $t^0$ are, respectively
\begin{eqnarray} 
&&\lambda^2 -\alpha^2\langle e_3,\gamma\rangle^2=0,\label{b11}\\
&&\lambda^2\Gamma^3-\left(\alpha\Gamma(\gamma',\beta,\gamma)-(\gamma',\beta,\gamma'')\right)^2=0.\label{b12}
\end{eqnarray}
Then   $\lambda=\pm\alpha\langle e_3,\gamma\rangle$  and $\lambda\Gamma^{3/2}=\pm (\alpha\Gamma(\gamma',\beta,\gamma)-(\gamma',\beta,\gamma''))$. Without loss of generality, we take the sign $+$, namely,
\begin{equation}\label{lb}
\lambda= \alpha\langle e_3,\gamma\rangle,\quad  \lambda\Gamma^{3/2}=  \alpha\Gamma(\gamma',\beta,\gamma)-(\gamma',\beta,\gamma''),
 \end{equation}
 and the reasoning in the other cases of sign is analogous.

 We now compute  the coefficient of $t^5$ of \eqref{b00}. We use \eqref{b11} and after some simplifications, we find
$$\alpha\langle e_3,\gamma\rangle\Big(\Theta+\alpha \langle e_3,\gamma\rangle\langle\gamma',\beta'\rangle+ \alpha (\gamma',\beta,\gamma)\Big)=0.$$
We use that $\lambda\not=0$. Because $\langle e_3,\gamma\rangle\not=0$ by \eqref{b11},  
$$\Theta+\alpha \langle e_3,\gamma\rangle\langle\gamma',\beta'\rangle+ \alpha (\gamma',\beta,\gamma)=0.$$
From here, we obtain an expression for $\Theta$, 
\begin{equation}\label{b13}
\Theta= -\alpha \langle e_3,\gamma\rangle\langle\gamma',\beta'\rangle-\alpha (\gamma',\beta,\gamma).
\end{equation}
Similarly, and for the coefficient for $t$ of \eqref{b00} and using \eqref{bbb} and \eqref{b12},  
\begin{eqnarray*}
&&\alpha\Gamma\langle e_3,\gamma\rangle-\langle e_3,\gamma''\rangle+3\alpha\Gamma^{1/2}\langle e_3,\gamma\rangle\langle\gamma',\beta'\rangle-2\alpha\langle\gamma',\beta'\rangle(\gamma',\beta,\gamma)+\Theta\langle\gamma',e_3\rangle=0,
\end{eqnarray*}
hence
$$
(\gamma',\beta,\gamma)=\frac{\alpha\Gamma\langle e_3,\gamma\rangle-\langle e_3,\gamma''\rangle+3\Gamma^{1/2}\langle e_3,\gamma\rangle\langle\gamma',\beta'\rangle+\Theta\langle\gamma',e_3\rangle}{2\alpha\langle\gamma',\beta'\rangle}.
$$
We now take the coefficient of $t^4$ in \eqref{b00}. This is a long expression that we simplify by replacing   the value   $(\gamma',\beta,\gamma)$ from the above equation, together  \eqref{b11} and \eqref{b13}. By vanishing this coefficient, we arrive to  
$$-3\alpha\langle e_3,\gamma\rangle^2\left(\Gamma^{1/2}+\langle\gamma',\beta'\rangle\right)^2=0.$$
Thus 
\begin{equation}\label{final}
\Gamma=\langle\gamma',\beta'\rangle^2.
\end{equation}
On the other hand,  since $ \langle\gamma',\beta \rangle=0$ and   from the basis, we have $\mathcal{B}$, $\gamma'=\langle\gamma',\beta'\rangle\beta'+\langle\gamma',e_3\rangle e_3$. Then  
$$\Gamma=|\gamma'|^2=\langle\gamma',\beta'\rangle^2+\langle\gamma',e_3\rangle^2.$$
 Combining with \eqref{final}, we deduce 
$\langle\gamma',e_3\rangle=0$, so  $\gamma'=\langle\gamma',\beta'\rangle\beta'$. Using the basis $\mathcal{B}$ again,  it is immediate from \eqref{bb} that  
$$(\gamma',\beta,\gamma)=-\langle\gamma,e_3\rangle\langle\gamma',\beta'\rangle.$$
Replacing in  \eqref{b13}, we deduce 
$\Theta=0$. This proves that $\beta(s)$ is a great circle of $\s^2$. Thus $e_3=\beta\times\beta'$ is a unit constant vector orthogonal to the plane $P$ containing $\beta$. From \eqref{lb}, it follows that 
$$\langle e_3,\gamma(s) \rangle=\frac{\lambda}{\alpha}$$
for all $s\in I$.   Finally, from the parametrization \eqref{ru}, we deduce
$$\langle X(s,t),e_3\rangle=\langle\gamma(s),e_3\rangle+t\langle\beta(s),e_3\rangle=\frac{\lambda}{\alpha},$$
proving that $X(s,t)$ is contained in a plane parallel to $P$. 
 
\end{enumerate}
After the discussion of the cases $\lambda=0$ and $\lambda\not=0$, and from Proposition \ref{pr1}, we conclude that the surface is a plane of $\r^3$. This completes the proof of Theorem \ref{t1}.

\section{Classification of translation surfaces}\label{sec3}

In this section we   study the solutions of \eqref{ll} (or equivalently of \eqref{sh4}) by the method of separation of variables. Let $M$ be the   graph of a function $u(x,y)=f(x)+g(y)$ where  $f:I\subset\r\rightarrow\r$ and $g:J\subset\r\rightarrow\r$ are smooth functions. If we parametrize by $X(x,y)=(x,y,f(x)+g(y))$, the set of points of the surface $M$ is  the sum of two planar curves, namely, 
\begin{equation}\label{pt}
X(x,y)=(x,0,f(x))+(0,y,g(y)).
\end{equation}
In the literature, surfaces of type $z=f(x)+g(y)$  are called    {\it translation surfaces} and they form part of a large family of ``surfaces d\'efinies pour des properti\'es cin\'ematiques'' following the terminology of Darboux in \cite{da}.   In case that one of the functions $f$ or $g$  is linear,  the surface is a ruled surface. Indeed, if for example, $g(y)=ay +b$ where $a,b\in\r$, then $\eta(x)=(x,0,f(x)+b)$ is the directrix of the surface and its parametrization is $X(x,y)=\eta(x)+y(0,1,a)$. This means that  $M$ is a ruled surface where all rulings are parallel to the fixed direction $(0,1,a)$, in particular, the surface is cylindrical.

 The proof of Theorem \ref{t2} is by contradiction. We assume   that both functions $f$ and $g$ are not linear. In particular, $f'f''\not=0$ and $g'g''\not=0$ in some subintervals   $\tilde{I}\subset I$ and $\tilde{J}\subset J$ respectively. Thus $f'f''g'g''\not=0$ in $\tilde{I}\times\tilde{J}$. 

We use  the     parametrization \eqref{pt} to calculate the Gauss map $N$ of $M$,  
$$N=\frac{X_x\times X_y}{|X_x\times X_y|}=\frac{(-f',-g',1)}{\sqrt{1+f'^2+g'^2}}.$$
Here, we denote by prime $'$ the derivative of $f$ or $g$ with respect to its variables. The mean curvature $H$ of $M$  is
$$H=\frac{(1+g'^2)f''+(1+f'^2)g''}{(1+f'^2+g'^2)^{3/2}}.$$ 
Then the self-similar solution equation \eqref{ll} is 
$$\frac{(1+g'^2)f''+(1+f'^2)g''}{(1+f'^2+g'^2)^{3/2}}=\alpha\frac{-xf'-y g'+f+g}{\sqrt{1+f'^2+g'^2}}+\lambda.$$
The determinant of the first fundamental for is $W=1+f'^2+g'^2$. Then the above equation can be expressed as
\begin{equation}\label{eq2}
(1+g'^2)f''+(1+f'^2)g''=\alpha(-xf'-y g'+f+g)\, W +\lambda\, W^{3/2}.
\end{equation}
The differentiation of (\ref{eq2}) with respect to the variable $x$ gives 
$$ \left(1+g'^2\right)f'''+2f'f''g''=-\alpha x f''\, W+2\alpha f'f''(-xf'-yg'+f+g)+3\lambda\, f'f''  W^{1/2}.$$
 A   differentiation of this equation with respect to the variable $y$ leads to
 $$2g'g''f'''+2f'f''g'''=-2\alpha x g'g''f''-2\alpha y f'f''g''+3\lambda f'f''g'g'' W^{-1/2},$$
 or equivalently, 
 \begin{equation}\label{eq4}
 2(f'''+\alpha x f'')g'g''+2(g'''+\alpha y g'')f'f''=3\lambda\, f'f''g'g''W^{-1/2}.
\end{equation}
 We separate the discussion in two cases according the constant $\lambda$.
\begin{enumerate}
\item Case $\lambda=0$.   We divide (\ref{eq4}) by $f'f''g'g''$, obtaining 
$$\frac{f'''+\alpha xf''}{f'f''}=-\frac{g'''+\alpha yg''}{g'g''}.$$
Since the left-hand side of this equation depends on the variable $x$, and the right-hand one on $y$, it follows that there is a constant $a\in\r$ such that
\begin{equation}\label{eq5}
\frac{f'''}{f'f''}+\alpha\frac{x}{f'}=-\frac{g'''}{g'g''}-\alpha\frac{y}{g'}=2a.
\end{equation}
From  a first integration of both equations, we find $m,n\in\r$ such that
\begin{equation}
\begin{split}\label{f1}
&f''+\alpha x f'-\alpha f=a f'^2+m,\\
&g''+\alpha y g'-\alpha g'=-a g'^2+n.
\end{split}
\end{equation}
By substituting into (\ref{eq2}), we obtain
$$ (n+a-\alpha f)f'^2+\alpha x f'^3=(a-m+\alpha g)g'^2+\alpha y g'^3-m-n.$$
Again, we deduce the existence of a constant $b\in\r$ such that 
\begin{equation}
\begin{split}\label{f2}
& (n+a-\alpha f)f'^2+\alpha x f'^3=b,\\
&(a-m+\alpha g)g'^2+\alpha y g'^3-m-n=b.
\end{split}
\end{equation}
We now give an argument with the function $f$ (it may done similarly for $g$). The function $f$ satisfies the first equation in \eqref{f1} and \eqref{f2}. 
Differentiating the first equation of \eqref{f2} with respect to $x$, it follows that
$$\left(2(n+a-\alpha f)+3\alpha x f'\right)f'f''=0.$$
Taking into account that $f'f''\not=0$, we deduce
$$2(n+a-\alpha f)+3\alpha x f'=0.$$
Instead to solve this equation, and in order to avoid the constants $a$ and $n$, we differentiate again this equation with respect to $x$. Simplifying,  we arrive to
$$f''=-\frac{1}{3x}f'.$$
The solution of this equation is $f(x)=cx^{2/3}+k$ where $c,k\in\r$.  Since $f$ is a not a constant function, then the constant $c$ is not $0$.   Once we have the expression of $f(x)$, we come back to   the first equation of \eqref{f1} and  we obtain
$$\frac{4 a c^2}{9} x^{-2/3}+\frac{1}{3} \alpha  c x^{2/3}+\frac{4 c}{9} x^{-4/3}+\alpha  k+m=0$$
for all $x\in I$. This equation is a polynomial equation on the function $x^{2/3}$. Then all coefficients vanish, in particular,   $c=0$, obtaining a contradiction.
\item Case $\lambda\not=0$. We divide \eqref{eq4} by $f'f''g'g''$, obtaining
$$\frac{2(f'''+\alpha x f'')}{f'f''}+\frac{2(g'''+\alpha y g'')}{g'g''}=3\lambda\frac{1}{\sqrt{1+f'^2+g'^2}}.$$
In view of the left-hand side of this equation is the sum of a function of $x$ with a function depending on $y$, if we differentiate with respect to $x$, and next with respect to $y$, the left-hand side vanishes. On the other hand, in the right-hand side, the same differentiations   give 
$$9\lambda\frac{f'f''g'g''}{(1+f'^2+g'^2)^{5/2}}=0.$$
This is a contradiction because $\lambda\not=0$ and  $f'f''g'g''\not=0$. This finishes the proof of Theorem \ref{t2}.

 \end{enumerate}
 
As a  final remark, we point out  that   the parametrization \eqref{pt} does not coincide with \eqref{kk} because for the translation surface \eqref{pt}  the rulings are not necessarily orthogonal to the plane containing the directrix $\eta(x)=(x,0,f(x)+b)$ (except if $a=0$), such it occurs in the parametrization \eqref{kk}.  If $a=0$ (and $b=0$),  equation \eqref{cee} is the equation \eqref{kk} for curves  $y=f(x)$. However, the cylindrical solutions given by Theorem \ref{t2} coincide, up to a linear isometry, with the ones given in Theorem \ref{t1}.


\small


\begin{thebibliography}{99} 



\bibitem{al} U. Abresch, J.  Langer,  The normalized curved shortening flow and homothetic
solutions, J. Differ. Geom., 23 (1986), 175--196.

\bibitem{aaa} D.J.  Altschuler, S.J.  Altschuler, S.B. Angenent, L.F.  Wu,  
The zoo of solitons for curve shortening in $R^n$, Nonlinearity 26 (2013),  1189--1226. 

\bibitem{an} H. Anciaux, Two non existence results for the self-similar equation in Euclidean 3-space, J. Geom., 96   (2009), 1--10.


\bibitem{ch} J-E. Chang,  One dimensional solutions of the $\lambda$-self shrinkers, Geom. Dedicata 
189 (2017), 97--112.

\bibitem{cow} Q-M.  Cheng, S. Ogata, G. Wei, Rigidity theorems of $\lambda$-hypersurfaces, Comm. Anal. Geom., 24 (2016),  45--58. 

\bibitem{cw0} Q-M. Cheng, G. Wei,  Compact embedded $\lambda$-torus in Euclidean spaces. arXiv:1512.04752, (2015).

\bibitem{cw} Q-M. Cheng, G. Wei, Complete $\lambda$-hypersurfaces of weighted volume-preserving mean curvature flow, 
Calc. Var. Partial Differential Equations 57 (2018), no. 2, Art. 32, 21 pp. 

\bibitem{cm} T. Colding, W. Minicozzi, 
  Generic mean curvature flow I; generic singularities, Ann. Math.,   175 (2012), 755--833.

\bibitem{da} G. Darboux,  Le\c{c}ons sur la Th\'eorie G\'en\'erale des Surfaces, 4 tomes. Gauthier Villars, Paris (1914).
 
 
\bibitem{lee} G. Drugan, H. Lee, X.H.  Nguyen, A survey of closed self-shrinkers with symmetry,  Results Math.,  73 (2018), 32pp.
 
 \bibitem{gr} M. Gromov, Isoperimetry of waists and concentration of maps, Geom. Func. Anal., 13 (2003), 178--215.

\bibitem{gu} Q. Guang, Gap and rigidity theorems of $ \lambda$-hypersurfaces, Proc. Amer. Math. Soc., 146 (2018), 4459--4471.

 \bibitem{ha} H.P. Halldorsson,   Self-similar solutions to the curve shortening flow, Trans. Amer. Math. Soc., 364 (2012), 5285--5309.
 
 \bibitem{ha2} H.P. Halldorsson,   Helicoidal surfaces rotating/translating under the mean curvature flow, Geom. Dedicata 162 (2013), 45--65.  
 \bibitem{il} T. Ilmanen,  Elliptic Regularization and Partial Regularity for Motion by Mean Curvature, Memoirs Amer. Math. Soc., vol. 520, 1994. 
 
\bibitem{km} S.J. Kleene, N.M. Moller, 
Self-shrinkers with a rotational symmetry, Trans. Amer. Math. Soc.,   366  (2014), 3943--3963.

\bibitem{mgr} M. McGonagle, J. Ross, 
The hyperplane is the only stable, smooth solution to the isoperimetric problem in Gaussian space, Geom. Dedicata 178 (2015), 277--296.

\bibitem{mo} F.  Morgan, Manifolds with density,  Notices Amer. Math. Soc., 52 (2005), 853--858.


\bibitem{ross} J. Ross, On the existence of a closed, embedded, rotational $\lambda$-hypersurface, J. Geom., 110  (2019),  26pp.

\bibitem{zfc} Y. Zhu, Y.  Fang, Q.  Chen, 
Complete bounded $\lambda$-hypersurfaces in the weighted volume-preserving mean curvature flow, 
Sci. China Math., 61 (2018),   929--942. 

\end{thebibliography}
\end{document}